\documentclass[12pt]{amsart}

\usepackage{hyperref,color}

\numberwithin{equation}{section}
\newtheorem{thm}{Theorem}[section]
\newtheorem{prop}[thm]{Proposition}
\newtheorem{cor}[thm]{Corollary}

\theoremstyle{remark}
\newtheorem{exmp}{Example}[section]
\newtheorem{rem}[thm]{Remark}

\newcommand{\E}{\mathcal{E}}

\newcommand{\B}{\mathcal{B}}

\newcommand{\F}{\mathtt{F}}

\newcommand{\Lloc}{L^1_{\mathrm{loc}}}
\newcommand{\ACloc}{AC}

\newcommand{\rg}{\mathrm{Im}}
\newcommand{\Ker}{\mathrm{Ker}}
\newcommand{\SX}{L^1}
\newcommand{\SXP}{L^1_{\varphi}}
\newcommand{\SXI}{L^{\infty}}

\begin{document}

\title[Substochastic semigroups and densities of PDMPs]
{Substochastic semigroups and densities of piecewise deterministic Markov
processes}

\author[M. Tyran-Kami\'nska]{Marta Tyran-Kami\'nska}
\address{Institute of Mathematics, Polish
Academy of Sciences and Institute of Mathematics, University of Silesia, Bankowa
14, 40-007 Katowice, Poland} \email{mtyran@us.edu.pl}

\subjclass[2000]{primary 47D06; secondary 60J25, 60J35, 60J75}%
\keywords{piecewise deterministic Markov process, stochastic semigroup,
strongly stable semigroup, fragmentation models}%

\maketitle

\begin{abstract}
Necessary and sufficient conditions are given for a substochastic semigroup on
$L^1$ obtained through the Kato--Voigt perturbation theorem to be either
stochastic or strongly stable. We show how such semigroups are related to
piecewise deterministic Markov process, provide a probabilistic interpretation of
our results, and apply them to fragmentation equations.
\end{abstract}

\section{Introduction}

Piecewise deterministic Markov processes (PDMPs) are Markov processes involving
deterministic motion punctuated by random jumps. A general theory for such
processes was introduced in \cite{davis84} within an abstract framework with
numerous examples from queueing and control theory. The sample paths $X(t)$ of the
PDMP depend on three local characteristics:   a flow $\pi$, a nonnegative jump
rate function $\varphi$, and a stochastic transition kernel $\mathcal{J}$. Instead
of a flow, here we consider semi-flows on a Borel state space $E$ such that
$\pi_t(E)\subseteq E$, $t\ge 0$, which leads to PDMPs without active boundaries
and allows us to use the more general formulation of stochastic models presented
in~\cite{jacod96}. Starting from $x$ the process follows the trajectory $\pi_t x$
until the first jump time $t_1$ which occurs at a rate $\varphi$. The value of the
process at the jump time $t_1$ is selected from the distribution
$\mathcal{J}(\pi_{t_1}x,\cdot)$ and the process restarts afresh from this new
point (see Section~\ref{sec:PDPc} for the construction). If the function $\varphi$
is unbounded then it might happen that the process is only defined up to a finite
random time, called an explosion time, so that we study the minimal PDMP with the
given characteristics.

Let the state space be a $\sigma$-finite measure space $(E,\E,m)$.
Suppose that the distribution of $X(0)$ is absolutely continuous
with respect to the measure $m$. One of our main objectives is to
give sufficient conditions for the distribution of $X(t)$ to be
absolutely continuous with respect to $m$ for all $t>0$, and to
derive rigorously an evolution equation for its density  $u(t,x)$.
This leads us to study equations of the form
\begin{equation}\label{eq:evd2}
\dfrac {\partial u(t,x) }{\partial t}=A_0u(t,x)-\varphi(x)u(t,x)+ P(\varphi
u(t,\cdot))(x),
\end{equation}
where $P$ is a stochastic operator on $\SX$ corresponding to the stochastic kernel
$\mathcal{J}$ (see Section~\ref{sec:pre}) and $A_0$ is the (infinitesimal)
generator of a strongly continuous semigroup of stochastic operators
(\emph{stochastic semigroup}) corresponding to the deterministic semi-flow $\pi$.
Let us write
\begin{equation}\label{eq:opAC}
Au=A_0u-\varphi u\quad \text{and}\quad  \mathcal{C}u=Au+P(\varphi u).
\end{equation}
When $\varphi$ is bounded, then the Cauchy problem associated with \eqref{eq:evd2}
is well posed, by the bounded perturbation theorem (see e.g.~\cite[Section
III.1]{engelnagel00}), and $\mathcal{C}$ generates a stochastic semigroup. If
$\varphi$ is unbounded, then $\mathcal{C}$ is the sum of two unbounded operators
and the existence and uniqueness of solutions to the Cauchy problem in $\SX$ is
problematic. The strategy which can be adapted to tackle such problems involves
perturbation results for strongly continuous semigroups of positive contractions
on $\SX$ (\emph{substochastic semigroups}). We refer the reader to the monograph
\cite{banasiakarlotti06} for an extensive overview on the subject. We make use of
one such result (Theorem~\ref{thm:kato} in Section~\ref{ssec:MSG}), which goes
back to~\cite{kato54} in the case of a discrete state space and was subsequently
developed in~\cite{voigt87,arlotti91,banasiak01}, from which it follows that the
operator $\mathcal{C}$ has an extension $C$ generating a substochastic semigroup
$\{P(t)\}_{t\ge 0}$ provided that the operator $A$ is the generator of a
substochastic semigroup on $\SX$ and $\mathcal{C}$ is defined on the domain
$\mathcal{D}(A)$ of~$A$. In general, the semigroup $\{P(t)\}_{t\ge 0}$ is
stochastic if and only if the generator $C$ is the minimal closed extension of
$(\mathcal{C},\mathcal{D}(A))$. In that case, if $u_0$ is nonnegative then  the
norm
\begin{equation*}
\|P(t)u_0\|=\int_E P(t)u_0(x) m(dx),\quad t\ge 0,
\end{equation*}
is constant in time, meaning that there is conservation of mass. If  $C$ is not
the minimal closed extension of  $(\mathcal{C},\mathcal{D}(A))$, then we have
\begin{equation}\label{eq:dishonet}
\|P(t)u_0\|<\|u_0\|
\end{equation} for some $u_0$ and $t>0$, meaning that there is
a loss of mass. Our objective is to study the two extreme cases: either
$\{P(t)\}_{t\ge0}$ is stochastic or it is \emph{strongly stable}
\[
\lim_{t\to\infty}\|P(t)u_0\|=0 \quad \text{for all}\quad u_0\in\SX.
\]
In Section~\ref{ssec:MSG} we provide general necessary and sufficient conditions
for either to hold (Theorems~\ref{prop:pert}--\ref{thm:opMarkov}). To the best of
our knowledge all past investigations of the semigroup $\{P(t)\}_{t\ge0}$
concentrated on providing necessary and sufficient conditions for conservation of
mass \cite{banasiakarlotti06} and it was only shown in~\cite{banasiakmokhtar05}
that if there is a loss of mass for fragmentation models and explosive birth-death
processes then \eqref{eq:dishonet} holds for every $u_0$ and sufficiently large
$t$.  Thus, the study of strong stability seems to be new.

Formulating the problem in the context of piecewise deterministic Markov processes
allows us to identify the corresponding semigroup $\{P(t)\}_{t\ge0}$ from a
probabilistic point of view (Theorem~\ref{thm:exist}). The combination of
probabilistic and fun\-ctio\-nal-analytic methods leads to rigorous results
providing a derivation of an evolution equation for densities of such processes
and necessary and sufficient conditions for the
 semigroup to be either stochastic or strongly stable.
In the discrete state space, \eqref{eq:evd2} with $A_0=0$ is the forward
Kolmogorov equation \cite{feller40} and  we recover the results of
\cite{kato54,reuter57}. To illustrate our general approach, we use fragmentation
models (Section~\ref{sec:frag}) in our framework and provide a refined analysis of
such models, previously studied extensively with either purely functional-analytic
or probabilistic methods \cite{melzak57,filippov61,mcgradyziff87,edwards90,
lamb97,haas03,fourniergiet03,
banasiaklamb03,arlottibanasiak04,arinorudnicki04,banasiak04,wagner05}. 
Our results can also be applied to stochastic differential equations with jumps
\cite{fourniergiet06,rudnicki07}.

The outline of this paper is as follows. In Section~\ref{sec:pre} we collect
relevant definitions for stochastic operators and give necessary and sufficient
conditions for strongly stable operators and semigroups. In~Section~\ref{ssec:MSG}
we recall the Kato--Voigt perturbation theorem and we prove necessary and
sufficient conditions for the corresponding semigroup to be either stochastic or
strongly stable. In Section~\ref{sec:exMs} we describe the extension techniques
introduced in~\cite{arlotti91}, and further developed in~\cite{banasiak01}, which
provide the characterization of the generator and the evolution equation for
densities. In Section \ref{sec:PDP} we study piecewise deterministic Markov
processes. In Section~\ref{sec:PDPc} we describe a general construction of PDMPs
and in Section~\ref{ssec:evol} the relation of such stochastic models to the
corresponding semigroups $\{P(t)\}_{t\ge 0}$ on $\SX$. In Section~\ref{sec:frag}
we let the operators $P$ and $A$ have definite forms and give a number of concrete
examples of situations that fit directly into our framework.

\section{Preliminaries}\label{sec:pre}
Let $(E,\E,m)$ be a $\sigma$-finite measure space and $L^p=L^p(E,\E,m)$ for all
$p\ge 1$. A linear operator $A\colon\mathcal{D}\to \SX$, where $\mathcal{D}$ is a
linear subspace of $\SX$, is said to be \emph{positive} if $Au\ge 0$ for $u\in
\mathcal{D}_+:=\mathcal{D}\cap \SX_+$. Then we write $A\ge 0$. Every positive
operator $A$ with $\mathcal{D}=\SX$ is a bounded operator. In general, we will
denote the domain of any operator $A$ by $\mathcal{D}(A)$, its range by $\rg(A)$,
$\rg(A)=\{Au:u\in \mathcal{D}(A)\}$, and its null space by $\Ker(A)$,
$\Ker(A)=\{u\in \mathcal{D}(A):Au=0\}$. The \emph{resolvent set} $\rho(A)$ of $A$
is the set of all complex numbers $\lambda$ for which $\lambda -A$ is invertible.
The family $R(\lambda,A):=(\lambda-A)^{-1}$, $\lambda\in\rho(A)$, of bounded
linear operators is called the resolvent of $A$. Finally, if $(A,\mathcal{D}(A))$
is the generator of a substochastic semigroup then $R(\lambda,A)u\ge R(\mu,A)u\ge
0$ for $\mu> \lambda>0$ and $u\in\SX_+$.

Let $D(m)\subset \SX$ be the set of all \emph{densities} on $E$, i.e.
\[
D(m)=\{u\in \, \SX: \,\, u\ge 0,\,\, \|u\|=1\},
\]
where $\|\cdot\|$ is the norm in $\SX$. A linear operator $P\colon \SX\to \SX$
such that $P(D(m))\subseteq D(m)$ is called \emph{stochastic} or
\emph{Markov}~\cite{almcmbk94}.

Let $\mathcal{J}\colon E\times \E\to[0,1]$ be a \emph{stochastic transition
kernel}, i.e. $\mathcal{J}(x,\cdot)$ is a probability measure for each $x\in E$
and the function $x\mapsto\mathcal{J}(x, B)$ is measurable for each $B\in\E$, and
let $P$ be a stochastic operator on $\SX$. If
\[
\int_E \mathcal{J}(x,B)u(x)m(dx)=\int_B Pu(x)m(dx)\quad \text{for all } B\in \E,
u\in D(m),
\]
then $P$ is called the \emph{transition} operator corresponding to $\mathcal{J}$.
If $p\colon E\times E\to[0,\infty)$ is a measurable function such that
\[
\int_E p(x,y) m(dx)=1,\quad  y\in E,
\]
then the operator $P$ defined by
\[
P u(x)= \int_E p(x,y)u(y) m(dy), \quad x\in E, u\in\SX,
\]
is stochastic  and it corresponds to the stochastic kernel
\[
\mathcal{J}(x,B)=\int_{B}p(y,x)m(dy),\quad x\in E, B\in \E.\] We simply say that
$P$ has \emph{kernel} $p$.

A linear operator $T$ on $\SX$ is called \emph{mean ergodic} if
\begin{equation*}
\lim_{N\to\infty}\frac{1}{N}\sum_{n=0}^{N-1} T^n u \quad \text{exists for all }
u\in \SX
\end{equation*}
and \emph{strongly stable} if
\begin{equation}\label{eq:pcL1}
\lim_{n\to\infty}\|T^n u\|=0\quad\text{for all } u\in\SX.
\end{equation}
Note that a stochastic operator is never strongly stable. We have the following
characterization of strongly stable positive contractions on $\SX$. The result
seems to be known but we cannot find appropriate references. We include its very
simple proof for the sake of completeness.

\begin{prop}\label{thm:mergeq} Let $T$ be a positive contraction on $\SX$
and  $T^*\colon \SXI\to \SXI$ be the adjoint of $T$. Then the following are
equivalent:
\begin{enumerate}
\item\label{it:li0} $T$ is mean ergodic and $\Ker(I-T)=\{0\}$.
\item\label{it:li1} $T$ is strongly stable.
\item\label{it:li2} Condition \eqref{eq:pcL1} holds for some $u\in \SX_+$, $u>0$ a.e.
\item\label{it:li3} If for some $f\in \SXI_+$ we have
$T^*f=f$ then $f=0$.
\item\label{it:li4}  $\lim\limits_{n\to\infty} T^{*n} 1=0$ a.e.
\end{enumerate}
\end{prop}
\begin{proof}
First observe that \eqref{it:li0} is equivalent to
\begin{equation*}\label{eq:mer0}
\lim_{N\to\infty}\frac{1}{N}\sum_{n=0}^{N-1} T^n u =0\quad\text{for all}\quad u\in
\SX.
\end{equation*}
Since $T$ is a positive contraction, the sequence $(\|T^n u\|)$ is convergent for
nonnegative $u$. Thus
\[
\lim_{N\to\infty}\frac{1}{N}\bigl\|\sum_{n=0}^{N-1} T^n
u\bigr\|=\lim_{N\to\infty}\frac{1}{N}\sum_{n=0}^{N-1}\bigl\| T^n
u\bigr\|=\lim_{n\to \infty}\|T^n u\|,
\]
by additivity of the norm, which gives 
\eqref{it:li0} $\Leftrightarrow$ \eqref{it:li1}. The implications \eqref{it:li4}
$\Rightarrow$ \eqref{it:li1}
and \eqref{it:li1} $\Rightarrow$ \eqref{it:li2} are trivial. 
Now assume that  \eqref{it:li2} holds.  Let $f\in L^\infty_+$ be such that
$T^*f=f$. We have
\[
\int_E  f u\,dm=\int_E T^{*n} f u \,dm=\int_E  f T^n u\,dm\le \|f\|_\infty \|T^n
u\|,
\]
which shows that $f=0$. Finally, assume that \eqref{it:li3} holds. Since $T^*1\le
1$, the limit $h:=\lim\limits_{n\to\infty} T^{*n} 1$ exists and $T^*h=h$. Thus
$h=0$ by \eqref{it:li3}.
\end{proof}

\begin{rem}
Note that if $T$ is a positive contraction with $\Ker(I-T)=\{0\}$ then $T$ is mean
ergodic if and only if $\Ker(I-T^*)=\{0\}$, by Sine's theorem~\cite{sine70}.
\end{rem}

We now state for later use the inheritance of mean ergodicity under domination.
This is a consequence of the Yosida-Kakutani ergodic theorem (see
e.g.~\cite[Theorem VIII.3.2]{yosida78}).

\begin{prop}\label{thm:uniq} Let $T$ and $K$ be positive contractions on $\SX$
such that
\[
Tu\le Ku\quad\text{for }u\in \SX_+.
\]
If $K$ is mean ergodic  then $T$ is mean ergodic.
\end{prop}

A semigroup $\{S(t)\}_{t\ge 0}$ is called \emph{strongly stable} if
\[
\lim_{t\to\infty}S(t)u=0\quad\text{for all } u\in \SX.
\]
Note that a stochastic semigroup is never strongly stable. The mean ergodic
theorem for semigroups \cite[Chapter VIII.4]{yosida78} and additivity of the norm
give the following characterization (see also \cite[Theorem 2.1 and Theorem
7.7]{chilltomilov07}).
\begin{prop}\label{l:markst} Let $\{S(t)\}_{t\ge 0}$ be a substochastic semigroup on $\SX$ with
generator $A$. Then the following are equivalent:
\begin{enumerate}
\item $\{S(t)\}_{t\ge 0}$ is strongly stable. \item For every $u\in\SX_+$
\begin{equation*}\label{eq:meAbel}
 \lim_{\lambda\downarrow
0}\lambda R(\lambda,A)u=0.
\end{equation*}
\item  $\rg(A)$ is dense in
$\SX$.
\end{enumerate}
\end{prop}

\section{Perturbation of substochastic semigroups}\label{ssec:MSG}

In this section we consider two linear operators $(A,\mathcal{D}(A))$ and
$(B,\mathcal{D}(B))$ in $\SX$ which are assumed throughout to have the following
properties:
\begin{enumerate}
  \item[(G1)] $(A,\mathcal{D}(A))$ generates a substochastic semigroup $\{S(t)\}_{t\ge
0}$;
  \item[(G2)] $\mathcal{D}(B)\supseteq \mathcal{D}(A)$ and $Bu\ge 0$ for $u\in
\mathcal{D}(A)_+$;
\item[(G3)] for every $u\in\mathcal{D}(A)_+$
\begin{equation}\label{eq:zero} \int_E (Au+Bu)\,dm=  0.
\end{equation}
\end{enumerate}
We refer to Sections \ref{sec:exMs} and~\ref{sec:frag} for examples of operators
satisfying (G1)--(G3).
\begin{thm}\label{thm:kato}\cite{kato54,voigt87,banasiak01}
There exists a  substochastic semigroup $\{P(t)\}_{t\ge 0}$ on $\SX$ generated by
an extension $C$ of the operator $(A+B,\mathcal{D}(A))$. The generator $C$ is
characterized by
\begin{equation}\label{eq:rp}
R(\lambda,C)u=\lim_{N\to\infty}R(\lambda,A)\sum_{n=0}^N (BR(\lambda,A))^nu, \quad
u\in \SX, \lambda>0,
\end{equation}
and $\{P(t)\}_{t\ge 0}$ is the smallest substochastic semigroup 
whose generator is an extension of $(A+B,\mathcal{D}(A))$.

Moreover, the following are equivalent:
\begin{enumerate}
\item  $\{P(t)\}_{t\ge 0}$ is a stochastic semigroup.
\item The generator $C$ is the closure of $(A+B,\mathcal{D}(A))$.
\item For some $\lambda>0$
\begin{equation}\label{eq:sspert}
\lim_{n\to\infty}\|(BR(\lambda,A))^n u\|=0\quad \text{for all}\quad u\in\SX.
\end{equation}
\end{enumerate}
\end{thm}

The semigroup $\{P(t)\}_{t\ge 0}$ from Theorem~\ref{thm:kato} can be obtained
\cite{banasiak01,banasiakarlotti06} as the strong limit in $\SX$ of semigroups
$\{P_r(t)\}_{t\ge 0}$ generated by $(A+rB,\mathcal{D}(A))$ as $r\uparrow 1$. It
satisfies the integral equation
\begin{equation}\label{eq:df}
P(t)u=S(t)u+\int_{0}^t P(t-s) B S(s)u\,ds
\end{equation}
for any $u\in \mathcal{D}(A)$ and $t\ge 0$, where $\{S(t)\}_{t\ge 0}$ is the
semigroup generated by $(A,\mathcal{D}(A))$, and it is also given by the
Dyson-Phillips expansion
\begin{equation}\label{eq:dpf1}
P(t)u=\sum_{n=0}^\infty S_n(t)u, \quad u\in \mathcal{D}(A),\;t\ge 0,
\end{equation}
where
\begin{equation}\label{eq:dpf2}
S_0(t)u=S(t)u,\quad S_{n+1}(t)u=\int_0^tS_{n}(t-s)B S(s)u\,ds, \quad n\ge 0.
\end{equation}

Let $\lambda>0$. Since the generator $C$ of the substochastic semigroup
$\{P(t)\}_{t\ge 0}$ is such that $Cu=(A+B)u$ for $u\in \mathcal{D}(A)$, we have
\begin{equation*}
(\lambda -C)R(\lambda,A)v=(\lambda -A-B)R(\lambda,A)v=(I-BR(\lambda,A))v
\end{equation*}
for $v\in \SX$.  Thus $\Ker(I-BR(\lambda,A))\subseteq \Ker(R(\lambda,A))$ and
\begin{equation}\label{eq:sum}
BR(\lambda,A)v +\lambda R(\lambda,A)v=v+(A+B)R(\lambda,A)v\quad\text{for } v\in
\SX_+.
\end{equation}
Combining this with (G2) and (G3), we obtain the following corollary.
\begin{cor}\label{r:sub}
Let $\lambda>0$. Then
\begin{equation}\label{eq:sum1}
\|BR(\lambda,A)u\|+\|\lambda R(\lambda,A)u\|=\|u\|\quad \text{for } u\in\SX_+
\end{equation}
and $BR(\lambda,A)$ is a positive contraction with $\Ker(I-BR(\lambda,A))=\{0\}$.
\end{cor}

\begin{rem}\label{rem:M}
Note that if $u\in\SX_+$ then for each $N\ge 0$
\begin{equation}\label{eq:normBR1}
\lambda\|R(\lambda, A)\sum_{n=0}^N (BR(\lambda,A))^nu\|=\|u\|-\|(BR(\lambda,
A))^{N+1}u\|.
\end{equation}
In fact, since $R(\lambda,A)v\in \mathcal{D}(A)_+$ for $v\in \SX_+$, we obtain, by
\eqref{eq:sum} and \eqref{eq:zero},
\begin{equation*}\label{eq:normBR}
\lambda\int_ER(\lambda, A)v dm=\int_E(v-BR(\lambda, A)v) dm,
\end{equation*}
which gives \eqref{eq:normBR1} for $v=\sum_{n=0}^N (BR(\lambda,A))^nu$.
\end{rem}

We have the following result for stochastic semigroups.
\begin{thm}
\label{prop:pert} Let $\lambda>0$. The following are equivalent:
\begin{enumerate}
\item\label{it:pert1} $\{P(t)\}_{t\ge 0}$ is a stochastic semigroup. 
\item\label{it:pert2} The operator $BR(\lambda,A)$ is mean ergodic.
\item\label{it:pert4}  $m\{x\in E:f_\lambda(x)>0\}=0$, 
where
\begin{equation}\label{def:fl}
f_\lambda(x)=\lim_{n\to\infty}(BR(\lambda,A))^{*n}1(x).
\end{equation}
\item\label{it:pert3} There is $u\in \SX_+$, $u>0$ a.e. such that 
\begin{equation*}\label{eq:css1}
\lim_{n\to\infty}\|(BR(\lambda,A))^nu\|=0.
\end{equation*}
\end{enumerate}
\end{thm}
\begin{proof}
By Corollary~\ref{r:sub}, the operator $BR(\lambda,A)$ is a positive contraction
with $\Ker(I-BR(\lambda,A))=\{0\}$. First assume that \eqref{it:pert1} holds.
Since the operator $\lambda R(\lambda,C)$ is stochastic, we have $\|\lambda
R(\lambda,C)u\|=\|u\|$ for $u\in\SX_+$. Hence \eqref{it:pert3} follows from
\eqref{eq:rp} and \eqref{eq:normBR1}. The implications \eqref{it:pert3}
$\Rightarrow$ \eqref{it:pert4} $\Rightarrow$ \eqref{it:pert2} $\Rightarrow$
\eqref{it:pert1} follow from Proposition~\ref{thm:mergeq} and
condition~\eqref{eq:sspert}.
\end{proof}

Next, we consider strong stability.
\begin{thm}\label{thm:sst}
The semigroup $\{P(t)\}_{t\ge 0}$ is strongly stable if and only if
\[
m\{x\in E:\liminf_{\lambda\downarrow 0}f_\lambda (x)<1\}=0,
\]
where $f_\lambda$ is defined in \eqref{def:fl}.
\end{thm}
\begin{proof}
It follows from \eqref{eq:normBR1} and the monotone convergence theorem that
\[
\|\lambda R(\lambda,C)u\|=\|u\|-\|f_\lambda u\|\quad \text{for}\quad u\in\SX_+.
\]
Since $f_\lambda\le 1$ for all $\lambda>0$ and $\|f_\mu u\|\le \|f_\lambda u\|$
for $\mu>\lambda$ and all $u\in\SX_+$, the claim follows from
Proposition~\ref{l:markst}.
\end{proof}

We now prove the following general result which provides another
sufficient condition for $\{P(t)\}_{t\ge 0}$ to be
stochastic.

\begin{thm}\label{thm:opMarkov}
Define the operator $K\colon \SX\to \SX$ by
\begin{equation}\label{eq:K0}
Ku=\lim_{\lambda\downarrow 0} BR(\lambda,A)u\quad \text{for } u\in\SX.
\end{equation}
Then the following hold:
\begin{enumerate}
\item\label{it:tm1} $K$ is a positive contraction.
\item\label{it:tm3} $K$ is stochastic  if and only if the semigroup
$\{S(t)\}_{t\ge 0}$ generated by $A$ is strongly stable.
\item\label{it:tm4} If $K$ is mean ergodic then $\{P(t)\}_{t\ge 0}$ is stochastic.
\end{enumerate}
\end{thm}
\begin{proof}
We have $\|B R(\lambda,A)u\|\le \|u\|$ for $u\in\SX_+$ and $0\le B R(\mu,A)u\le B
R(\lambda,A)u$ for $\mu>\lambda$, $u\in\SX_+$. Thus the limit
$\lim_{\lambda\downarrow 0}B R(\lambda,A)u$ exists and $\|\lim_{\lambda\downarrow
0}B R(\lambda,A)u\|=\lim_{\lambda\downarrow 0}\|B R(\lambda,A)u\|$ for
$u\in\SX_+$, by the monotone convergence theorem. Since the cone $\SX_+$ is
generating, i.e.~$\SX=\SX_+-\SX_+$, $K$ is a well defined positive contraction.
 From \eqref{eq:sum1} it follows that
\[
\|K u\|=\|u\|-\lim_{\lambda\downarrow 0}\lambda\|R(\lambda,A)u\|\quad \text{for }
u\in\SX_+,
\]
which implies \eqref{it:tm3}, by Proposition~\ref{l:markst}. Since $B
R(\lambda,A)\le K$ for $\lambda>0$, claim \eqref{it:tm4} is a consequence of
Proposition~\ref{thm:uniq} and Theorem~\ref{prop:pert}.
\end{proof}

The semigroup $\{P(t)\}_{t\ge 0}$ dominates $\{S(t)\}_{t\ge 0}$. By
part~\eqref{it:tm3} of~Theorem~\ref{thm:opMarkov}, we obtain 
 the
following necessary condition for $\{P(t)\}_{t\ge 0}$ to be strongly stable.

\begin{cor}
If the semigroup $\{P(t)\}_{t\ge 0}$ is strongly stable then the operator $K$
defined by \eqref{eq:K0} is stochastic.
\end{cor}

\section{Evolution equation}\label{sec:exMs}

In this section we introduce an abstract setting in which the evolution equations
for densities of PDMPs can be studied. Let $P$ be a stochastic operator on $\SX$,
$\varphi\colon E\to[0,\infty)$ be a measurable function,  and let
\[
\SXP=\{u\in \SX:\int_E\varphi(x)|u(x)|m(dx)<\infty\}.
\]
We assume that $\{S(t)\}_{t\ge 0}$ is a substochastic semigroup on $\SX$ with
generator $(A,\mathcal{D}(A))$ such that
\begin{equation}\label{eq:defA}
\mathcal{D}(A)\subseteq \SXP\quad \text{and}\quad \int_EAu\, dm=-\int_E\varphi u\,
dm\quad \text{for}\quad u\in \mathcal{D}(A)_+.
\end{equation}
\begin{rem}\label{r:genA}
Note that \eqref{eq:defA} holds if
\begin{equation*}\label{e:A}
Au=A_0u-\varphi u\quad \text{for}\quad u\in \mathcal{D}(A)\subseteq
\mathcal{D}(A_0)\cap \SXP,
\end{equation*}
where $(A_0,\mathcal{D}(A_0))$ is the generator of a stochastic semigroup.
\end{rem}

Define the operator $B$ by $Bu=P(\varphi u)$, $u\in \SXP$. Since $P$ is positive
and $\|P(\varphi u)\|=\|\varphi u\|$ for $u\in \mathcal{D}(A)_+$, it follows from
\eqref{eq:defA} that the operators $(A,\mathcal{D}(A))$ and $(B,\SXP)$ satisfy the
assumptions (G1)--(G3) of Section~\ref{ssec:MSG} and, by Theorem~\ref{thm:kato},
there exists a smallest substochastic semigroup $\{P(t)\}_{t\ge 0}$ with generator
$(C,\mathcal{D}(C))$ which is an extension of the operator
\begin{equation}\label{eq:defC}
\mathcal{C}u=Au+P(\varphi u)\quad \text{for}\quad u\in\mathcal{D}(A).
\end{equation}
Since $(C,\mathcal{D}(C))$ is the generator of $\{P(t)\}_{t\ge 0}$, the Cauchy
problem
\[
u'(t)=C u(t),\quad t\ge 0, \quad u(0)=u_0,
\]
possesses a unique classical solution for all $u_0\in\mathcal{D}(C)$, which is
given by $u(t)=P(t)u_0\in \mathcal{D}(C)$. However, as we do not know the operator
$C$, we should rather work with the equation
\begin{equation*}
u'(t)=\mathcal{C} u(t),\quad \text{where}\quad  \mathcal{C} u=\mathcal{A} u + P(\varphi u) 
\end{equation*}
and  $\mathcal{A}$ and $P$ are extensions of the operators $A$ and $P$ such that
$\mathcal{D}(C)\subseteq \mathcal{D}(\mathcal{C})$. The existence of such
extensions follows from the construction of \cite[Section 2]{arlotti91} (see
\cite[Section 6.3]{banasiakarlotti06} for more details) which we now reformulate
in terms of the operators that appear in \eqref{eq:defC}.

We denote by $L=L(E,\E,m)$ the space of equivalent classes of all measurable
$[-\infty,\infty]$-valued function on $E$ and by $L^0$ the subspace of $L$
consisting of all elements which are finite almost everywhere. If $0\le u_n\le
u_{n+1}$, $u_n\in\SX$, $n\in\mathbb{N}$, then the pointwise almost everywhere
limit of $u_n$ exists and will be denoted by $\sup_n u_n$, so that $\sup_n u_n\in
L$. If $T$ is a positive bounded linear operator, it may be extended pointwise and
linearly beyond the space $\SX$ in the following way: if  $u\in L_+$ then we
define
\[
Tu=\sup_n Tu_n \quad \text{for }  u=\sup_n u_n, u_n\in\SX_+
\]
(note that $Tu$ is independent of the particular approximating sequence $u_n$),
and if $u\in L$ is such that $T|u|\in L^0$ then we set $Tu=Tu_{+}-Tu_{-}$. Since
$R(1,A)$ and $P$ are positive contractions, they have pointwise extensions, which
will be denoted  in what follows by $R(1,A)$ and $P$.

Let
\[
\F=\{u\in L: R(1,A)|u|\in\SX\}\quad\text{and}\quad R_1u=R(1,A)u \quad \text{for }
u\in \F.
\]
Then $\F\subset L^0$ and the operator $R_1\colon\F\to L^1$ is
one-to-one~\cite[Lemma 3.1]{banasiak01}. We can
 define the operator $\mathcal{A}\colon\mathcal{D}(\mathcal{A})\to L^0$ by
\begin{equation}\label{def:A}
\mathcal{A}u=u-R_1^{-1}u\quad \text{for } u\in
\mathcal{D}(\mathcal{A}):=\{R_1v:v\in\F\}
\end{equation}
and the operator $\mathcal{B}\colon \mathcal{D}(\mathcal{B})\to L^0$ by
\begin{equation*}
\mathcal{B}u=P(\varphi u)\quad \text{for } u\in \mathcal{D}(\mathcal{B}):=\{u\in
\SX: P(\varphi |u|)\in L^0\}.
\end{equation*}
Since $\SXP\subset \mathcal{D}(\mathcal{B})$ and $\mathcal{A}$ is
an extension of $(A,\mathcal{D}(A))$,   the operator
$\mathcal{C}\colon\mathcal{D}(\mathcal{C})\to \SX$ given by
\begin{equation*}\label{eq:opMax}
\mathcal{C}u=\mathcal{A}u+\mathcal{B}u\quad \text{for } u\in
\mathcal{D}(\mathcal{C})=\{u\in \mathcal{D}(\mathcal{A})\cap
\mathcal{D}(\mathcal{B}): \mathcal{C}u\in \SX\}
\end{equation*}
is an extension of the operator $(\mathcal{C},\mathcal{D}(A))$ defined by
\eqref{eq:defC}. Theorem~1 of~\cite{arlotti91} characterizes the generator
$(C,\mathcal{D}(C))$ of the semigroup $\{P(t)\}_{t\ge 0}$ in the following way:
\[
Cu=\mathcal{C}u \quad \text{for } u\in
\mathcal{D}(C)=\{u\in\mathcal{D}(\mathcal{C}):\lim_{n\to\infty}\|(R_1
\mathcal{B})^n u\|=0\}.
\]
Since $(C,\mathcal{D}(C))$ is a closed extension of
$\mathcal{C}_{|\mathcal{D}(A)}$, we obtain
\[
\mathcal{D}(\overline{\mathcal{C}_{|D(A)}})\subseteq \mathcal{D}(C)\subseteq
\mathcal{D}(\mathcal{C}).
\]
Consequently, if $u_0\in \mathcal{D}(C)\cap D(m)$ then the equation
\begin{equation}\label{eq:evd1}
u'(t)=\mathcal{C} u(t),\quad t\ge 0,\quad u(0)=u_0,
\end{equation}
has a nonnegative strongly differentiable solution  $u(t)$ which is given by
$u(t)=P(t)u_0$ for $t\ge 0$ and if $\{P(t)\}_{t\ge0}$ is stochastic then this
solution is unique in $D(m)$. Recall that $D(m)$ is the set of densities.

\begin{rem}
Suppose that the operator $P$ has kernel $p$. Then for every $u\in\SX_+$ we obtain
\[
P(\varphi u)(x)=\int_{E} p(x,y)\varphi(y)u(y)m(dy),
\]
by the monotone convergence theorem.

If $Au=-\varphi u$ for $u\in\SXP$ then
\[
\F=\{u\in L^0: \frac{u}{1+\varphi}\in\SX\}\quad\text{and}\quad
\mathcal{A}u=-\varphi u\quad \text{for}\quad u\in\mathcal{D}(\mathcal{A})=\SX.
\]
\end{rem}

\section{Piecewise deterministic Markov processes}\label{sec:PDP}

\subsection{Construction}\label{sec:PDPc}
Let $E$ be a Borel subset of a Polish space (separable complete metric space) and
let $\B(E)$ be the Borel $\sigma$-algebra. We consider three local characteristics
$(\pi,\varphi,\mathcal{J})$:
\begin{enumerate}
\item A \emph{semidynamical system} $\pi\colon\mathbb{R}_+\times E\to E$ on
$E$, i.e. $\pi_0x=x$, $\pi_{t+s}x=\pi_t(\pi_sx)$ for  $x\in E$, $s,t\in\mathbb{R}_+$,  
and the mapping $(t,x)\mapsto \pi_tx$ is continuous \cite[Section 7.2]{almcmbk94}. 
\item\label{it:jrf}
A \emph{jump rate function} $\varphi\colon E\to\mathbb{R}_+$ which is Borel
measurable and  such that for every $x\in E$, $t>0$, the function $s\mapsto
\varphi(\pi_sx)$ is integrable on $[0,t)$. We additionally assume that
\begin{equation}\label{eq:phi2}
\lim_{t\to\infty} \int_{0}^{t}\varphi(\pi_s x)ds=+\infty\quad \text{for all } x\in
E.
\end{equation}
\item\label{it:skJ} A \emph{jump distribution}
$\mathcal{J}\colon E\times \B(E)\to[0,1]$ which is a stochastic transition kernel
such that $\mathcal{J}(x,\{x\})=0$ for all $x\in E$.
\end{enumerate}
The local characteristics $(\pi,\varphi,\mathcal{J})$ determine a piecewise
deterministic Markov process $\{X(t)\}_{t\ge 0}$ (PDMP) on $E$ (see
e.g.~\cite{davis84,davis93,jacod96}). Define the function
\[
\Phi_x(t)=1-e^{-\phi_x(t)},\quad t>0, x\in E, \quad \text{where }
\phi_x(t)=\int_{0}^{t}\varphi(\pi_s x)ds.
\]
From \eqref{it:jrf} it follows that for every $x\in E$ the function $\phi_x$ is
non-decreasing and right-continuous, because $\phi_x(\tau)\to 0$ as
$\tau\downarrow 0$ and $\phi_x(t+\tau)=\phi_{\pi_t x}(\tau)+\phi_x(t)$ for all
$t,\tau > 0$, $x\in E$. This and \eqref{eq:phi2} imply that $\Phi_x$ is the
distribution function of a positive finite random variable. Let
$\phi_x^{\leftarrow}$ be the \emph{generalized inverse} of
 $\phi_ x$, i.e.
\[
\phi_x^{\leftarrow}(q)=\inf\{t: \phi_x(t)\ge q\}, \quad q\ge 0,
\]
and let $\kappa\colon [0,1]\times E\to E$ be a measurable function such that
\begin{equation}\label{d:stkj}
\mathcal{J}(x,B)=l_1\{q\in[0,1]:\kappa(q,x)\in B\}\quad \text{for } x\in E,
B\in\B(E),
\end{equation}
where $l_1$ is the Lebesgue measure on $([0,1],\B([0,1]))$; the existence of this
function follows from \eqref{it:skJ} and the regularity of the space $E$
\cite[Lemma 3.22]{kallenberg02}. Observe that if $\vartheta$ is a random variable
uniformly distributed on $(0,1)$, then  $\kappa(\vartheta,x)$ has distribution
$\mathcal{J}(x,\cdot)$ and if $\varepsilon$ is
exponentially distributed with mean $1$, then
$\phi_x^{\leftarrow}(\varepsilon)$ has distribution $\Phi_x$ (note that
$\varepsilon=-\log(1-\vartheta)$).

Let $\varepsilon_n, \vartheta_n$, $n\in \mathbb{N}$, be a sequence of independent
random variables, where the $\varepsilon_n$ are exponentially distributed with
mean $1$ and the $\vartheta_n$ are uniformly distributed on $(0,1)$. Let $\Delta
t_0=\tau$, $\tau\in\mathbb{R}_+$, and let $\xi_0=x$, $x\in E$. Define recursively
the sequence of \emph{holding times} as
\begin{equation*}
\Delta t_n:=\phi_{\xi_{n-1}}^{\leftarrow}(\varepsilon_n),
\end{equation*}
and \emph{post-jump positions} as
\begin{equation*}
\xi_{n}:=\kappa(\vartheta_n,\pi_{\Delta t_n}(\xi_{n-1})).
\end{equation*}
Then $(\xi_n,\Delta t_n)$ is a discrete time-homogeneous Markov process on
$E\times\mathbb{R}_+$ with stochastic transition kernel given by
\[
\mathcal{G}((x,\tau),  B\times[0,t))=\int_{0}^t
\mathcal{J}(\pi_sx,B)\varphi(\pi_sx)e^{-\int_{0}^s\varphi(\pi_rx)dr}ds
\]
for $(x,\tau)\in E\times\mathbb{R}_+$, $t\in\mathbb{R}_+$, and
$B\in\B(E)$. Let $\mathbb{P}_{(x,\tau)}$ be the distribution of
$(\xi_n,\Delta t_n)$ starting at $(\xi_0,\Delta t_0)=(x,\tau)$. We
write in an abbreviated fashion $\mathbb{P}_x$ for
$\mathbb{P}_{(x,0)}$ and $\mathbb{E}_x$ for the integration with
respect to $\mathbb{P}_x$, $x\in E$.

Now let $\Delta t_0\equiv 0$ and define \emph{jump times} as
\begin{equation*}\label{eq:explt}
t_n:=\sum_{l=0}^n \Delta t_l\quad \text{for } n\ge 0.
\end{equation*}
Since $\Delta t_n>0$ for all $n\ge 1$ with probability one, the sequence $(t_n)$
is increasing and we can introduce the \emph{explosion time}
\begin{equation*}
t_{\infty}:=\lim_{n\to \infty}t_n.
\end{equation*}
The sample path of the process $\{X(t)\}_{t\ge 0}$ starting at $X(0)=\xi_0=x$ is
now defined by
\begin{equation*}\label{eq:xit}
X(t)=\left\{
  \begin{array}{ll}
  \pi_{t-t_{n}}(\xi_{n}), & \text{if }t_{n}\le t<t_{n+1}, n\ge 0,\\
\Delta, &\text{if } t\ge t_\infty,
\end{array}
\right.
\end{equation*}
where $\Delta\notin E$ is some extra state representing a cemetery point for
$\{X(t)\}_{t\ge 0}$. The process $\{X(t)\}_{t\ge 0}$ is called the \emph{minimal
PDMP} corresponding to the characteristics $(\pi,\varphi,\mathcal{J})$. It has
right continuous sample paths, by construction, and it is a strong Markov process,
by \cite[Theorem 8]{jacod96}. The process is called \emph{non-explosive}  if
$\mathbb{P}_x(t_\infty=\infty)=1$ for all $x\in E$.

In particular, if $\pi_tx=x$ for all $t\ge 0, x\in E$, then $\{X(t)\}_{t\ge 0}$ is
the so-called  \emph{pure jump Markov} process. Observe that in this case
condition \eqref{eq:phi2} is equivalent to $\varphi(x)>0$ for every $x\in E$ and
$\mathbb{P}_x(t_\infty=\infty)=1$ is equivalent to
\begin{equation*}
\sum_{n=1}^\infty \frac{\varepsilon_n}{\varphi(\xi_{n-1})}=\infty \quad
\mathbb{P}_x-\text{a.e.}
\end{equation*}
We also have $\mathbb{P}_x(t_\infty=\infty)=1$ if and only if the
series $\sum_{n=1}^\infty \frac{1}{\varphi(\xi_{n-1})}$ diverges $
\mathbb{P}_x-$a.e. (see e.g. \cite[Proposition
12.19]{kallenberg02}). General sufficient conditions for the
explosion of pure jump Markov processes are contained in
\cite[Section 2]{wagner05}. Note also that pure jump Markov
processes on a
countable set $E$ are continuous-time Markov chains. 

\subsection{Existence of densities for PDMP}\label{ssec:evol}

Let $\{X(t)\}_{t\ge 0}$ be the minimal PDMP on $E$ with characteristics
$(\pi,\varphi,\mathcal{J})$ as defined in Section~\ref{sec:PDPc} and let $m$ be a
$\sigma$-finite measure on $\E=\B(E)$. In this section we impose further
restrictions on the characteristics $(\pi,\varphi,\mathcal{J})$ which allow us to
define a substochastic semigroup $\{P(t)\}_{t\ge 0}$ on $\SX$ corresponding to the
Markov process $\{X(t)\}_{t\ge 0}$ and to provide a probabilistic characterization
of the analytic results from Section~\ref{ssec:MSG}.

We assume that a stochastic operator $P\colon \SX\to \SX$ is the transition
operator corresponding to $\mathcal{J}$ and that a substochastic semigroup
$\{S(t)\}_{t\ge 0}$ on $\SX$, with generator $(A,\mathcal{D}(A))$
satisfying~\eqref{eq:defA}, is such that
\begin{equation}\label{eq:T0St}
\int_E e^{-\int_{0}^{t}\varphi(\pi_rx)dr}1_B(\pi_tx)u(x)\,m(dx)=\int_B
S(t)u(x)\,m(dx)
\end{equation}%
for all $t\ge 0$, $u\in \SX_+$, $ B\in\B(E)$. As shown in Section~\ref{sec:exMs},
there exists a smallest substochastic semigroup $\{P(t)\}_{t\ge 0}$ on $\SX$ whose
generator is an extension of the operator $(\mathcal{C},\mathcal{D}(A))$ defined
in~\eqref{eq:defC}. The semigroup $\{P(t)\}_{t\ge 0}$ will be referred to as the
\emph{minimal semigroup on $\SX$ corresponding to} $(\pi,\varphi,\mathcal{J})$.

\begin{rem}
Observe that from \eqref{eq:T0St} it follows that for every $t>0$ the
transformation $\pi_t\colon E\to E$ is nonsingular, i.e. $m(\pi_t^{-1}(B))=0$ for
all $B\in \B(E)$ such that $m(B)=0$ \cite[Section 3.2]{almcmbk94}, and that there
is a stochastic operator $P_0(t)$ on $\SX$ satisfying
\[
\int_E 1_B(\pi_tx)v(x)\,m(dx)=\int_B P_0(t)v(x)\,m(dx),\quad B\in\B(E), v\in \SX.
\]
Hence, \[ S(t)u=P_0(t)v_t,\quad \text{where
}v_t(x)=e^{-\int_{0}^{t}\varphi(\pi_rx)dr}u(x).\] If $\{P_0(t)\}_{t\ge 0}$ is a
stochastic semigroup with generator $A_0$ then one may expect that the minimal
semigroup $\{P(t)\}_{t\ge 0}$ solves \eqref{eq:evd2} and that the operator $A$ is
as in~\eqref{eq:opAC} (see Section~\ref{sec:frag} for some examples).
\end{rem}

Substituting $B=E$ into \eqref{eq:T0St} leads to
\[
\|S(t)u\|=\int_{E}e^{-\int_0^t \varphi(\pi_rx)dr} u(x)\,m(dx)\quad\text{for all
 }u\in \SX_+,
\]
which shows that $\{S(t)\}_{t\ge 0}$ is strongly stable if and only if
condition~\eqref{eq:phi2} holds. Thus, the operator $K\colon \SX\to \SX$ defined
by
\begin{equation}\label{eq:K}
Ku=\lim_{\lambda\downarrow 0} P(\varphi R(\lambda,A)u)\quad \text{for } u\in\SX
\end{equation}
is stochastic, by Theorem~\ref{thm:opMarkov}.

 The main result of
 this section is the following (we use the convention $e^{-\infty}=0$).

\begin{thm}\label{thm:exist} Let $(t_n)$ be the
sequence of jump times and
 $t_\infty=\lim_{n\to\infty} t_n$ be the explosion time for $\{X(t)\}_{t\ge 0}$.
Then the following hold:
\begin{enumerate}
\item\label{thm:exist:i1}
For any $\lambda>0$
\begin{equation*}
\lim_{n\to\infty}(P(\varphi R(\lambda,A)))^{*n}1(x)=\mathbb{E}_x(e^{-\lambda
t_\infty})\quad m- a.e. \;x.
\end{equation*}
\item\label{thm:exist:i2} For  any $B\in\B(E)$,  $u\in\mathcal{D}(A)_+$, and $t>0$
\begin{equation*}
\int_{B}P(t)u(x)m(dx)=\int_E \mathbb{P}_{x}(X(t)\in B,t<t_\infty)u(x)m(dx).
\end{equation*}
\item\label{thm:exist:i3} The operator  $K$ as defined in \eqref{eq:K}
is the transition operator corresponding to the discrete-time Markov process
$(X(t_n))_{n\ge 0}$ with stochastic kernel
\begin{equation*}\label{eq:Kkernel}
\mathcal{K}(x,B) =\int_{0}^\infty
\mathcal{J}(\pi_sx,B)\varphi(\pi_sx)e^{-\int_{0}^s\varphi(\pi_rx)dr}ds, \quad x\in
E, B\in\B(E).
\end{equation*}
\end{enumerate}
\end{thm}
\begin{proof} Let $M(E)_+$ (respectively $BM(E)_+$) be the space of
all (bounded) Borel measurable nonnegative functions on $E$.  From \eqref{eq:T0St}
we obtain, by approximation,
\begin{equation}\label{eq:T0St1}
\int_E e^{-\int_{0}^{t}\varphi(\pi_rx)dr}f(\pi_tx)u(x)\,m(dx)=\int_E
f(x)S(t)u(x)\,m(dx)
\end{equation}%
for all $t\ge 0$, $u\in \SX_+$, $ f\in M(E)_+$.  Let $\lambda>0$ and
\begin{equation*}
G^\lambda f(x)=\int_{0}^\infty e^{-\lambda s}T_0(s)(\varphi \mathcal{J
}f)(x)\,ds\quad x\in E,\; f\in BM(E)_+,
\end{equation*}
where the operators $\mathcal{J }$ and $T_0(s)$ are defined by
\begin{equation*}\label{e:QJ}
\mathcal{J}f(x)=\int_E f(y)\mathcal{J}(x,dy),\quad x\in E, f\in BM(E)_+,
\end{equation*}
and
\begin{equation*}\label{e:T0}
T_0(s)f(x)=e^{-\int_{0}^{s}\varphi(\pi_rx)dr}f(\pi_s x), \quad x\in E, f\in
M(E)_+, s\ge 0.
\end{equation*}

From \eqref{eq:T0St1} and Fubini's theorem it follows that
\begin{equation}\label{e:GP}
\int_E G^\lambda f(x) u(x)m(dx)=\int_E f(x)P(\varphi R(\lambda,A)u)(x)m(dx)
\end{equation}
for $f\in BM(E)_+, u\in \SX_+$,  which gives $ (P(\varphi
R(\lambda,A)))^{*}f=G^\lambda f. $ On the other hand, the construction of the
sequence $(t_n,X(t_n))$ yields
\begin{equation}\label{eq:Glam}
(G^\lambda)^n f(x)=\mathbb{E}_x(f(X(t_n))e^{-\lambda t_n} ),\quad x\in E,
n\in\mathbb{N}, f\in BM(E)_+,
\end{equation}
which, by the monotone convergence theorem, leads to
\[
\lim_{n\to\infty}(G^\lambda)^n 1(x)=\mathbb{E}_x(e^{-\lambda t_\infty})
\]
and proves \eqref{thm:exist:i1}.

In order to show \eqref{thm:exist:i2}, for each $n\ge 0$ we define
\[
T_n(t) f(x)=\mathbb{E}_x f(X(t))1_{\{t<t_{n+1}\}},\quad  x\in E, t\in\mathbb{R}_+,
f\in BM(E)_+.
\]
Let $B\in \B(E)$ and $u\in \mathcal{D}(A)_+$. From the construction of the process
and the strong Markov property it follows that \cite[Theorem 9]{jacod96}
\begin{equation*}
T_n(t)1_B(x)=T_0(t)1_B(x)+ \int_0^{t} T_0(s)(\varphi
\mathcal{J}(T_{n-1}(t-s)1_B))(x)ds
\end{equation*}
for all $x\in E$, $t\ge 0$, and $n\ge1$. Hence, by induction, 
\[
\int_E T_n(t)1_B(x)u(x)m(dx)=\int_{B}\sum_{j=0}^n S_j(t)u(x)m(dx), \quad n\ge 0,
t>0,
\]
where the $S_j$ are defined in \eqref{eq:dpf2}. From \eqref{eq:dpf1} we obtain
\[
\lim_{n\to\infty} \int_{B}\sum_{j=0}^n S_j(t)u(x)m(dx)=\int_{B}P(t)u(x)m(dx).
\]
On the other hand,
\[
T_n(t)1_B(x)=\mathbb{P}_x(X(t)\in B,t<t_n)\uparrow \mathbb{P}_x(X(t)\in
B,t<t_\infty)
\]
for all $x\in E$, which proves~\eqref{thm:exist:i2}.

Finally, from \eqref{eq:Glam} we conclude that
\[\lim_{\lambda\downarrow 0} G^\lambda
1_B(x)=\mathbb{E}_x(1_B(X(t_1)))=\mathcal{K}(x,B),\] which completes the proof of
\eqref{thm:exist:i3}, by \eqref{e:GP}.
\end{proof}

As a direct consequence of Theorem~\ref{thm:exist} and Theorem~\ref{prop:pert} we
obtain the following corollary.

\begin{cor}\label{cor:Ms}
The semigroup $\{P(t)\}_{t\ge 0}$ is  stochastic if and only if
\[
m\{x\in E: \mathbb{P}_x(t_\infty<\infty)>0\}=0.
\]
In that case, if the distribution of $X(0)$ has a density $u_0\in \mathcal{D}(A)$
then $X(t)$ has the density $P(t)u_0$ for all $t>0$.
\end{cor}

Furthermore, as a consequence of Theorem~\ref{thm:exist} and Theorem~\ref{thm:sst}
we obtain the following result.

\begin{cor}\label{cor:ss}
The semigroup $\{P(t)\}_{t\ge 0}$ is strongly stable if and only if
\[
m\{x\in E: \mathbb{P}_x(t_\infty=\infty)>0\}=0.
\]
\end{cor}

\begin{rem}\label{rem:essest}
Note that for every density $u\in \mathcal{D}(A)_+$ we obtain
\begin{equation}\label{eq:essest}
\int_E P(t)u(x)m(dx)=\int_E \mathbb{P}_x(t_\infty>t)u(x)m(dx)\quad \text{for all }
t>0,
\end{equation}
by part~\eqref{thm:exist:i2} of Theorem~\ref{thm:exist}. In particular, if
$\mathcal{D}(A)$ is such that for every $u\in \SX_+$ we can find a non-decreasing
sequence $u_n\in \mathcal{D}(A)_+$ such that $u_n\uparrow u$ then
\eqref{eq:essest} holds for all $u\in \SX_+$.
\end{rem}

\begin{cor}
Let $E$ be a countable set, $m$ be the counting measure on $E$, $\varphi >0$, and
$\{X(t)\}_{t\ge 0}$ be a pure jump Markov process on $E$. Then the semigroup
$\{P(t)\}_{t\ge 0}$ is stochastic if and only if the process $\{X(t)\}_{t\ge 0}$
is non-explosive.
\end{cor}

\section{Fragmentation models revisited}\label{sec:frag}
In this section we illustrate the applicability of our results to fragmentation
models described by linear rate equations \cite{mcgradyziff87,edwards90, lamb97,
banasiaklamb03,arinorudnicki04,banasiak04}. For a recent survey of analytic
methods for such models we refer the reader to~\cite{banasiak06}. See also
\cite{bertoin06} for a different probabilistic treatment of so-called random
fragmentation processes.

Let $E=(0,\infty)$, $\E=\B(E)$, and $m(dx)=xdx$. Let $b\colon E\times
E\to\mathbb{R}_+$ be a Borel measurable function such that for every $y>0$
\begin{equation}\label{eq:mass}
\int_0^yb(x,y)xdx=y \quad\text{and}\quad b(x,y)=0\quad \text{for} \quad x\ge y.
\end{equation}
The stochastic kernel defined by
\begin{equation}\label{eq:Jfr}
\mathcal{J}(x,B)=\frac{1}{x}\int_{0}^x 1_B(y) b(y,x)ydy \quad \text{for } x\in E,
B\in \B(E),
\end{equation}
will be referred to as the \emph{fragmentation kernel}.  According to
\eqref{d:stkj}, we have $\mathcal{J}(x,B)=l_1\{q\in[0,1]:\kappa(q,x)\in B\}$,
where
\begin{equation*}
\kappa(q,x)=H_x^{\leftarrow}(q)x\quad \text{for}\quad q\in[0,1], x>0,
\end{equation*}
and $H_x^{\leftarrow}(q)=\inf\{r\in[0,1]: H_x(r)\ge q\}$, $q\in [0,1]$, is the
generalized inverse of the distribution function
\begin{equation*}
H_x(r)=\int_{0}^r b(xz,x)xzdz \quad \text{for } r\in[0,1].
\end{equation*}
Note that  $0<H_x^{\leftarrow}(q)\le 1$ for all $x$ and $q\in(0,1)$.

The kernel $\mathcal{J}$ is called \emph{homogenous} if $b$ is of the form
\begin{equation}\label{eq:hk}
b(x,y)=\frac{1}{y}h\left(\frac{x}{y}\right) \quad \text{for}\quad 0<x<y,
\end{equation}
where $h\colon (0,1)\to\mathbb{R}_+$ is a Borel measurable function with $
\int_{0}^1h(z)zdz=1. $ Since $H_x(r)$ does not depend on $x$, we obtain
\begin{equation*}
\kappa(q,x)=H^{\leftarrow}(q)x, \quad \text{where}\quad H(r)=\int_{0}^r h(z)zdz.
\end{equation*}

The kernel $\mathcal{J}$ is called \emph{separable} if $b$ is of the form
\begin{equation*}\label{eq:hs}
b(x,y)=\frac{\beta(x)y}{\Lambda(y)}\quad \text{for}\quad x<y, \quad
\text{where}\quad \Lambda(y)=\int_{0}^{y}\beta(z)zdz
\end{equation*}
and $\beta$ is a nonnegative Borel measurable function on $E$ such that
$\Lambda(y)$ is finite and positive for all $y>0$. We have
$H_x(r)=\Lambda(xr)/\Lambda(x)$ for $r\in[0,1]$. Hence
$H_x^{\leftarrow}(q)=\Lambda^{\leftarrow}(q\Lambda(x))/x$ and in this case
\[
\kappa(q,x)=\Lambda^{\leftarrow}(q\Lambda(x)).
\]
Since $\Lambda(\Lambda^\leftarrow x)=x$ for all $x>0$, the mapping $x\mapsto
\Lambda(x)$ transforms this case into the homogenous fragmentation with $H(r)=r$
for $r\in(0,1)$.

In what follows we assume that $\varepsilon_n,\vartheta_n$, $n\in\mathbb{N}$, and
$\xi_0$ are independent random variables, where the $\varepsilon_n$ are
exponentially distributed with mean $1$, the $\vartheta_n$ are uniformly
distributed on $(0,1)$, and $\xi_0$ is an $E-$valued random variable.

\subsection{Pure fragmentation}
In this section we consider the pure fragmentation equation
\cite{mcgradyziff87,lamb97,melzak57}
\begin{equation}\label{eq:fr1}
\dfrac {\partial u(t,x)}{\partial t}=\int_{x}^\infty
b(x,y)\varphi(y)u(t,y)dy-\varphi(x)u(t,x),\; t>0, x>0,
\end{equation}
where $b$ satisfies \eqref{eq:mass} and $\varphi$ is a positive Borel measurable
function. If we let $\psi(y,x)=b(x,y)\varphi(y)$ then \eqref{eq:fr1} has the same
form as in \cite{melzak57} in the absence of coagulation. For a discussion of the
model we refer the reader to \cite[Chapter 8]{banasiakarlotti06}.

We rewrite equation \eqref{eq:fr1} in the form \eqref{eq:evd1} with the stochastic
operator $P$ on $\SX$ given by
\begin{equation}\label{eq:MoPf}
Pu(x)=\int_x^\infty b(x,y)u(y)dy, \quad  u\in \SX,
\end{equation}
and \[ Au=-\varphi u,\quad u\in\SXP=\{u\in
\SX:\int_{0}^\infty\varphi(x)|u(x)|xdx<\infty\}.
\]
Observe that $P$ is the transition operator corresponding to $\mathcal{J}$ as
defined in~\eqref{eq:Jfr}. The operator $(A,\SXP)$ generates a substochastic
semigroup $\{S(t)\}_{t\ge 0}$ on $\SX$ where $S(t)u(x)=e^{-\varphi(x)t}u(x)$,
$t\ge 0,x\in E$, $u\in\SX$. Hence, \eqref{eq:T0St} holds with $\pi_tx=x$, $t\ge
0,x\in E$.

Let $\{X(t)\}_{t\ge 0}$ be the minimal pure jump Markov process with
characteristics $(\pi,\varphi,\mathcal{J})$ and let $\{P(t)\}_{t\ge 0}$ be the
minimal semigroup on $\SX$ corresponding to $(\pi,\varphi,\mathcal{J})$ as defined
in Section~\ref{ssec:evol}. The sequences of jump times $t_n$ and post-jump
positions $\xi_n=X(t_n)$ satisfy
\begin{equation*}\label{eq:xinjm}
t_n=\sum_{k=1}^n\frac{\varepsilon_k}{\varphi(\xi_{k-1})},\quad
\xi_n=H_{\xi_{n-1}}^{\leftarrow}(\vartheta_n) \xi_{n-1}, \quad n\ge 1.
\end{equation*}
Since the sequence $(\xi_n)$ is non-increasing, we can take $\Delta=0$ and write
for the explosion time
\[
t_\infty=\inf\{t>0: X(t)=0\}.
\]
As a consequence of Corollary~\ref{cor:Ms} we obtain the following result of
\cite{lamb97}.
\begin{cor}\label{cor:Mpj}
If $\varphi$ is bounded on bounded subsets of $(0,\infty)$ then $\{P(t)\}_{t\ge
0}$ is stochastic.
\end{cor}
\begin{proof}
Let $N>0$ and let $M_N<\infty$ be such that $\varphi(x)\le M_N$ for all $x\le N$.
Thus, if $\xi_0\le N$ then $\xi_k\le N$ for all $k$, and $t_n\ge \sum_{k=1}^n
\varepsilon_k/M_N$ for all $n$. As a result $\mathbb{P}_x(t_\infty<\infty)=0$ for
all $x\le N$, and the claim follows from Corollary~\ref{cor:Ms}.
\end{proof}

From Corollary~\ref{cor:ss} we obtain the following result.
\begin{cor}\label{cor:sspj}
Let $V$ be a nonnegative Borel measurable function such that $V(x)\varphi(x)\ge 1$
for all $x>0$. If
\[
m\{x\in E: \mathbb{P}_x\Bigl(\sum_{n=1}^\infty \varepsilon_n
V(\xi_n)=\infty\Bigr)>0\}=0
\] then $\{P(t)\}_{t\ge 0}$ is strongly
stable.
\end{cor}

\begin{exmp} Consider a homogenous kernel as in~\eqref{eq:hk} and let $V(x)=x^\gamma/a$, where $\gamma,a>0$.
The random variable
\[
\tau=\sum_{k=1}^\infty \varepsilon_k
\prod_{l=1}^{k-1}H^{\leftarrow}(\vartheta_l)^\gamma
\]
 is finite with probability $1$, by \cite[Theorem 1.6]{vervaat79}. Thus,
if $\varphi(x)\ge  a /x^\gamma$ for $x>0$, then $\{P(t)\}_{t\ge 0}$ is strongly
stable, by Corollary~\ref{cor:sspj}. Since $t_\infty\le V(\xi_0)\tau$, we have for
every $u\in\SX_+$, by Remark~\ref{rem:essest},
\begin{equation}\label{eq:ssest}
\int_{0}^\infty P(t)u(x)xdx\le \int_{0}^\infty (1-F_\tau(atx^{-\gamma}))u(x)xdx
\quad \text{for all } t>0,
\end{equation}
with equality when $\varphi(x)=  a /x^\gamma$,   where $F_\tau$ is the
distribution function of $\tau$.

In particular, if $h(z)=(\nu+2)z^\nu$ with $\nu+2>0$, then
$H^{\leftarrow}(\vartheta_1)=\vartheta_1^{1/(\nu+2)}$ and $\tau$ has the gamma
distribution \cite[Example 3.8]{vervaat79} such that 
\[
1-F_\tau(q)=\frac{1}{\Gamma(1+(\nu+2)/\gamma)}\int_{q}^\infty
s^{(\nu+2)/\gamma}e^{-s}ds, \quad q\ge 0,
\]
where $\Gamma$ is the Gamma function.
When 
$\varphi(x)=  1 /x^\gamma$ the equality in \eqref{eq:ssest} coincides with the
heuristic results of \cite{mcgradyziff87}; for values of $\nu$ and $\gamma$ such
that $(\nu+2)/\gamma\in\mathbb{N}\cup \{0\}$ we obtain
\[
\int_{0}^\infty P(t)u(x)xdx=\int_{0}^\infty
e^{-tx^{-\gamma}}\sum_{k=0}^{(\nu+2)/\gamma} \frac{(tx^{-\gamma})^k}{k!}u(x)xdx
\]
for all $t>0$ and $u\in \SX_+$. See \cite[Example 6.5]{banasiak06} for quite
involved calculations for the specific choice of $\nu=0$ and $\gamma=1$.
\end{exmp}

\begin{rem}
Since the sequence $(\xi_n)$ is non-increasing, it converges with probability one
to some random variable. In particular, when the kernel is either homogenous or
separable the limiting random variable is zero. Then it is sufficient  to look
only at a neighborhood of zero to decide whether the semigroup is stochastic or
not.
\end{rem}

\subsection{Fragmentation with growth}
Pure fragmentation, described by \eqref{eq:fr1}, may occur together with other
phenomena. In this section we study fragmentation processes with continuous
growth, where the growth process is described by a semidynamical system $\pi$
satisfying the equation
\begin{equation}\label{e:grpr}
\frac{\partial}{\partial t}\pi_tx=g(\pi_tx) \quad \text{for }x,t>0,
\end{equation}
where $g$ is a strictly positive continuous function. We refer the reader to
\cite{arinorudnicki04,banasiak04,mackeytyran08} for related examples. We denote by
$\Lloc$  the space of all Borel measurable functions on $E$ which are integrable
on compact subsets of $E$ and by $\ACloc$ the space of absolutely continuous
functions on $E$.

Our first task is to construct the minimal PDMP $\{X(t)\}_{t\ge0}$ on $E$ with
characteristics $(\pi,\varphi,\mathcal{J})$, where $\pi$ satisfies~\eqref{e:grpr},
$\varphi\in \Lloc$ is nonnegative, and $\mathcal{J}$ is the fragmentation kernel
\eqref{eq:Jfr}. We assume throughout this section that there is $\bar{x}>0$ such
that
\begin{equation}\label{eq:asGQ}
\int_{\bar{x}}^\infty \frac{1}{g(z)}dz=\infty\quad \text{and }\quad
\int_{\bar{x}}^\infty\frac{\varphi(z)}{g(z)}dz=\infty.
\end{equation}
Since $1/g\in \Lloc$ and $\varphi/g\in \Lloc$, we can define
\begin{equation}\label{eq:GQ}
G(x)=\int_{x_0}^x \frac{1}{g(z)}dz\quad \text{and} \quad
Q(x)=\int_{x_1}^{x}\frac{\varphi(z)}{g(z)}dz,
\end{equation}
where $x_0=0$ and $x_1=0$ when the integrals exist for all $x$ and, otherwise,
$x_0$, $x_1$ are any points in $E$.

The function $G$ is increasing, invertible, continuously differentiable on $E$,
and the formula $r(t,x)=G^{-1}(G(x)+t)$ defines a  monotone continuous function in
each variable. Since $G(\infty)=+\infty$, the function $r(t,x)$ is well defined
for all $t\ge 0$, $x\in E$ and determines a semidynamical system on $E$
\begin{equation}\label{d:rtx}
\pi_t x=G^{-1}(G(x)+t).
\end{equation}
In the case when $G(0)=-\infty$ the function $r(t,x)$ is well defined for all
$t\in\mathbb{R}$ and $x\in E$, so that we have, in fact, a flow $\pi_t$ on $E$
such that $\pi_t(E)=E$. In any case, for any given $x>0$ we have
$\pi_{-t}x=r(-t,x)\in E$ for all $t>0$ such that $t<G(x)-G(0)$.

The function $Q$ is non-decreasing. Let $Q^{\leftarrow}$ be the generalized
inverse of $Q$, which is defined and finite for all $q\in\mathbb{R}$,
by~\eqref{eq:asGQ}. We have
\[
\phi_x(t)=\int_{0}^{t}\varphi(\pi_rx)dr=\int_{x}^{\pi_tx}\frac{\varphi(z)}{g(z)}dz=Q(\pi_tx)-Q(x)
\quad \text{for } x>0, t\ge 0,
\]
so that \eqref{eq:phi2} holds if and only if $Q(\infty)=\infty$, which is our
assumption \eqref{eq:asGQ}. From \eqref{d:rtx} it follows that
\begin{equation}\label{eq:inphi}
\phi_x^{\leftarrow}(q)=G(Q^{\leftarrow}(Q(x)+q))-G(x)
\end{equation}
and
\begin{equation}\label{eq:inpi}
\pi_{\phi_x^{\leftarrow}(q)}x=Q^{\leftarrow}(Q(x)+q)\quad \text{for } x>0,q\ge 0.
\end{equation}
Consequently, the random variables  $t_n$ and $\xi_n=X(t_n)$, $n\ge 1$, now
satisfy
\begin{equation}\label{eq:xinjmg}
t_n=\sum_{k=1}^n\phi_{\xi_{k-1}}^{\leftarrow}(\varepsilon_k),\quad
\xi_n=H_{Q^{\leftarrow}(Q(\xi_{n-1})+\varepsilon_n)}^{\leftarrow}(\vartheta_n)
Q^{\leftarrow}(Q(\xi_{n-1})+\varepsilon_n).
\end{equation}

\begin{rem}\label{rem:Mi0}
If $\varphi$ is bounded above by a constant $a$ then
\[
\phi_x^{\leftarrow}(q)=G(Q^{\leftarrow}(Q(x)+q))-G(x)\ge
\frac{1}{a}(Q(Q^{\leftarrow}(Q(x)+q))-Q(x))\ge \frac{q}{a}.
\]
Thus $t_n\ge \frac{1}{a}\sum_{k=1}^n\varepsilon_k$ for every $n$, so that
$\mathbb{P}_x(t_\infty<\infty)=0$ for all $x>0$.
\end{rem}
\begin{rem}\label{rem:Mi}
Observe that if \[ m\{x\in E:\mathbb{P}_x(\limsup_{n\to\infty}\xi_n<\infty)>0\}=0,
\]
then $m\{x\in E:\mathbb{P}_x(t_\infty<\infty)>0\}=0$. This is a consequence of
$G(\infty)=\infty$ and $ t_n\ge
G(Q^{\leftarrow}(Q(\xi_{n-1})+\varepsilon_n))-G(\xi_0)$ for $n\ge 1$.
\end{rem}

We now turn our attention to the minimal semigroup $\{P(t)\}_{t\ge 0}$ on $\SX$
corresponding to $(\pi,\varphi,\mathcal{J})$. For $t>0$ we define the operators
$S(t)$ on $\SX$ by
\begin{equation}\label{eq:dst}
S(t)u(x)=\mathbf{1}_E(\pi_{-t}x)u(\pi_{-t}x)\frac{\pi_{-t}xg(\pi_{-t}x)}{xg(x)}e^{Q(\pi_{-t}x)-Q(x)},\quad
x\in E,
\end{equation}
for  $u\in \SX.$ Then $\{S(t)\}_{t\ge 0}$ is a substochastic semigroup on $\SX$
satisfying \eqref{eq:T0St}, whose generator is of the form
\begin{equation*}\label{eq:genA}
Au(x)=-\frac{1}{x}\frac{d}{dx}\bigl(xg(x)u(x)\bigr)-\varphi(x)u(x),\quad
u\in\mathcal{D}(A)=\mathcal{D}_0\cap \SXP,
\end{equation*}
where $u\in \mathcal{D}_0$  if and only if the function $\tilde{u}(x)=xg(x)u(x)$
is such that $\tilde{u}\in AC$, $\tilde{u}'(x)/x $ belongs to $\SX$, and,
additionally  $\lim_{x\to 0}\tilde{u}(x)=0$ when $G(0)=0$. This can be derived
from \cite[Theorem 5]{mackeytyran08} by an isomorphic transformation of the space
$\SX$. The resolvent operator $R(1,A)$ is given by
\[
R(1,A)u(x)=\frac{1}{xg(x)}e^{-G(x)-Q(x)}\int_{0}^x e^{G(y)+Q(y)} u(y)ydy,\quad
u\in\SX,
\]
and its extension $R_1$, as described in Section~\ref{sec:exMs},
is defined by the same integral expression. It can be proved as in
\cite[Lemma 4.1]{arlottibanasiak04} that the extension
$(\mathcal{A},\mathcal{D}(\mathcal{A}))$ defined in \eqref{def:A}
is of the form
\[
\mathcal{A}u(x)=-\frac{1}{x}\frac{d}{dx}\bigl(xg(x)u(x)\bigr)-\varphi(x)u(x)
\]
for $u\in \mathcal{D}(\mathcal{A})\subseteq\{u\in\SX: \tilde{g}u\in \ACloc\}$,
where $\tilde{g}(x)=xg(x)$, $x>0$. Consequently, the corresponding evolution
equation on $\SX$ for $u(t,x)=P(t)u_0(x)$ is of the form
\begin{equation*}
\dfrac {\partial u(t,x)}{\partial t}=-\frac{1}{x}\frac{\partial}{\partial
x}(xg(x)u(t,x))-\varphi(x)u(t,x)+\int_{x}^\infty b(x,y)\varphi(y)u(t,y)dy
\end{equation*}
with $u(0,x)=u_0(x)$.

Finally, we apply our results from Sections~\ref{ssec:MSG} and~\ref{ssec:evol} to
the minimal semigroup $\{P(t)\}_{t\ge 0}$. It is easily seen, by
Theorem~\ref{thm:exist}, that the operator $K$ as defined in \eqref{eq:K} is a
stochastic operator with  kernel
\begin{equation}\label{e:k1}
k(x,y)=\int_{\max\{x,y\}}^\infty b(x,z)\frac{\varphi(z)}{z g(z)}e^{Q(y)-Q(z)}dz,
\quad x,y\in (0,\infty).
\end{equation}
By Theorem~\ref{thm:opMarkov}, $\{P(t)\}_{t\ge 0}$ is stochastic, if $K$ is mean
ergodic. In particular, $K$ is mean ergodic, if $K$ is \emph{asymptotically
stable}, i.e. there is $u_*\in D(m)$ such that $Ku_*=u_*$ and
\[
\lim_{n\to\infty}\|K^n u-u_*\|=0\quad \text{for all }u\in D(m).
\]
General sufficient conditions for the latter to hold are contained in
\cite[Chapter 5]{almcmbk94} and we have the following result.

\begin{thm}\cite[Theorem 5.7.1]{almcmbk94}\label{thm:ast}
If the kernel $k$ satisfies
\[
\int_{0}^\infty \inf_{0<y<r} k(x,y) m(dx) >0\quad \text{for every }r>0
\]
and has a Lyapunov function $V\colon(0,\infty)\to[0,\infty)$, i.e.
$\lim_{x\to\infty}V(x)=\infty$ and for some constants $0\le c<1$, $d\ge 0$ 
\[
\int_{0}^\infty V(x)Ku(x) m(dx)\le c \int_{0}^\infty V(x)u(x)m(dx)+d\quad\text{for
} u\in D(m),
\]
then the operator $K$ is asymptotically stable.
\end{thm}

In the reminder of this section, we will study the semigroup $\{P(t)\}_{t\ge 0}$
under the assumption that the kernel $\mathcal{J}$ is homogeneous as
in~\eqref{eq:hk}.

\begin{cor}\label{thm:examp}
Assume that there are $r,\gamma>0$ such that
\begin{equation}\label{eq:ccc}
\int_{0}^r \frac{\varphi(z)}{g(z)}dz<\infty\quad\text{and}\quad  \liminf_{x\to
\infty}\frac{\varphi(x)}{x^{\gamma -1}g(x)}>0.
\end{equation}
Then the operator $K$ is asymptotically stable and the semigroup $\{P(t)\}_{t\ge
0}$ is stochastic.
\end{cor}
\begin{proof}
Since $Q(0)=0$, we have
\[
k(x,y)\ge \int_{x}^\infty h\Bigl(\frac{x}{z}\Bigr)\frac{\varphi(z)}{z^2
g(z)}e^{-Q(z)}dz\quad \text{for } 0<y<r<x,
\]
which shows that the first condition in Theorem~\ref{thm:ast} holds,  by
\eqref{eq:ccc}. We also have
 $\lim_{x\to\infty}x^\gamma e^{-Q(x)}=0$ and
\[
\int_{0}^\infty x^\gamma k(x,y)xdx=\int_{0}^1 z^\gamma zh(z)dz
\Bigl(y^\gamma+\gamma e^{Q(y)}\int_{y}^\infty z^{\gamma -1}e^{-Q(z)}dz\Bigr)\le c
y^\gamma+d,
\]
where $c:=\int_{0}^1 z^\gamma zh(z)dz<1=\int_{0}^1 zh(z)dz$, which shows that
$V(x)=x^\gamma$ is a Lyapunov function.
\end{proof}

The assumptions in Corollary~\ref{thm:examp} cannot be essentially weakened.

\begin{exmp}
Suppose that $\varphi(x)=g(x)/x$ for all $x>0$. Then $Q(x)=\log x$ for $x>0$, so
that the sequence $\xi_n$ is of the form
\[
\xi_n=\xi_0\prod_{k=1}^nH^{\leftarrow}(\vartheta_k)e^{\varepsilon_k}\quad\text{for
} n\ge 1.
\]
Let $\mu_0=\int_{0}^1 \log zh(z)z dz$. Observe that $\mu_0$ is always negative and
might be equal to $-\infty$.  If $\mu_0\ge -1$ then for any $x$ we have
$\mathbb{P}_x(\limsup_{n\to\infty}\xi_n=\infty)=1$. If $\mu_0<-1$ then, by the
strong law of large numbers, $\mathbb{P}_x(\lim_{n\to\infty}\xi_n=0)=1$ for
all~$x$. Thus $K$ is not asymptotically stable in both cases and $\{P(t)\}_{t\ge
0}$ is stochastic when $\mu_0\ge -1$, by Remark~\ref{rem:Mi}. In any case, if
$g(x)\le \tilde{a} x$ then $\{P(t)\}_{t\ge 0}$ is stochastic, by
Remark~\ref{rem:Mi0}.
\end{exmp}

\begin{exmp}
Suppose that $g(x)=x^{1-\beta}$ and $\varphi(x)=ax^\alpha$ for $x>0$, where $a>0$.
Then condition \eqref{eq:asGQ} holds if and only if  $\beta\ge 0$ and
$\alpha+\beta\ge 0$. If $\alpha+\beta>0$ then \eqref{eq:ccc} holds, thus
$\{P(t)\}_{t\ge 0}$ is stochastic. Now suppose that $\alpha+\beta=0$. If either
$\beta=0$ or $\mu_0=\int_{0}^1 \log zh(z)zdz\ge -1/a $ then $\{P(t)\}_{t\ge 0}$ is
also stochastic, as in the preceding example. If $\beta>0$ and $\mu_0<-1/a$ then
$\{P(t)\}_{t\ge 0}$ is strongly stable. This follows from the representation
\[
t_n=\frac{1}{\beta}\sum_{k=1}^n (e^{\beta
\varepsilon_k/a}-1)\xi_{k-1}^{\beta}=\frac{\xi_0^\beta}{\beta}\sum_{k=1}^n
(e^{\beta \varepsilon_k}-1)\prod_{l=1}^{k-1}H^{\leftarrow}(\vartheta_l)^\beta
e^{\beta\varepsilon_l/a}
\]
and the fact that the random variable
\[
\tau=\sum_{k=1}^\infty (e^{\beta
\varepsilon_k/a}-1)\prod_{l=1}^{k-1}H^{\leftarrow}(\vartheta_l)^\beta
e^{\beta\varepsilon_l/a}
\]
is finite with probability $1$ precisely when $\mu_0<-1/a$ \cite[Theorem
1.6]{vervaat79}.
\end{exmp}

\subsection{Fragmentation with decay}
In this section we consider fragmentation processes with continuous degradation,
where the degradation process is described by a semidynamical system $\pi$
satisfying the equation
\begin{equation}\label{e:depr}
\frac{\partial}{\partial t}\pi_tx=-g(\pi_tx) \quad \text{for }x,t>0,
\end{equation}
where $g$ is a strictly positive continuous function. Now the corresponding linear
evolution equation will be of the form
\begin{equation*}
\dfrac {\partial u(t,x)}{\partial t}=\frac{1}{x}\frac{\partial}{\partial
x}(xg(x)u(t,x))-\varphi(x)u(t,x)+\int_{x}^\infty b(x,y)\varphi(y)u(t,y)dy.
\end{equation*}
 We refer the reader to
\cite{edwards90,banasiaklamb03,arlottibanasiak04} for related examples.

To construct the minimal PDMP $\{X(t)\}_{t\ge 0}$ with characteristics
$(\pi,\varphi,\mathcal{J})$, where $\pi$ satisfies~\eqref{e:depr}, $\varphi\in
\Lloc$ is nonnegative, and $\mathcal{J}$ is given by~\eqref{eq:Jfr}, we redefine
the functions $G$ and $Q$ from \eqref{eq:GQ} in such a way that formulas
\eqref{d:rtx}--\eqref{eq:xinjmg} remain valid. We assume that there is $\bar{x}>0$
such that
\begin{equation}\label{eq:asGQd}
\int^{\bar{x}}_0 \frac{1}{g(z)}dz=\infty\quad \text{and }\quad
\int^{\bar{x}}_0\frac{\varphi(z)}{g(z)}dz=\infty,
\end{equation}
and  define
\begin{equation}\label{eq:asGQd1}
G(x)=\int_x^{x_0} \frac{1}{g(z)}dz\quad \text{and}\quad Q(x)=\int_x^{x_1}
\frac{\varphi(z)}{g(z)}dz,
\end{equation}
where $x_0=+\infty$ and $x_1=+\infty$ when the integrals exist for all $x$, and,
otherwise, $x_0,x_1$ are any points from $E$.   The semidynamical system $\pi$
defined by \eqref{d:rtx} satisfies~\eqref{e:depr}. The function $G$ is now
decreasing and $Q$ is non-increasing. As the generalized inverse of $Q$ we take
\[
Q^{\leftarrow}(q)=\left\{
                    \begin{array}{ll}
                      \sup\{x:Q(x)\ge q\}, &  q>Q(\infty),\\
                      0, & q\le Q(\infty)\quad\text{and}\quad  Q(\infty)>-\infty.
                    \end{array}
                  \right.
\]
With these alterations equations \eqref{eq:inphi} and \eqref{eq:inpi} remain
valid. The random variables $t_n$ and $\xi_n=X(t_n)$, $n\ge 1$, again satisfy
\eqref{eq:xinjmg}.

The semigroup $\{S(t)\}_{t\ge 0}$, defined by \eqref{eq:dst}, is a substochastic
semigroup on $\SX$ satisfying \eqref{eq:T0St}, whose generator is of the form
\begin{equation*}\label{eq:genAd}
Au(x)=\frac{1}{x}\frac{d}{dx}\bigl(xg(x)u(x)\bigr)-\varphi(x)u(x),\quad
u\in\mathcal{D}(A)=\mathcal{D}_0\cap \SXP,
\end{equation*}
where $u\in \mathcal{D}_0$  if and only if the function $\tilde{u}(x)=xg(x)u(x)$
is such that $\tilde{u}\in AC$, $\tilde{u}'(x)/x $ belongs to $\SX$, and,
additionally  $\lim_{x\to \infty}\tilde{u}(x)=0$ when $G(\infty)=0$. This can be
derived from \cite[Theorem 7]{mackeytyran08}.

We conclude this section with the following characterization of the minimal
semigroup $\{P(t)\}_{t\ge 0}$ on $\SX$ corresponding to
$(\pi,\varphi,\mathcal{J})$.
\begin{cor}\label{cor:Mfd}

If $\varphi$ is bounded on bounded subsets of $(0,\infty)$ then $\{P(t)\}_{t\ge
0}$ is stochastic.

If $V$ is a non-decreasing function such that $V(x)\varphi(x)\ge 1$ for all $x>0$
and
\[
m\{x\in E: \mathbb{P}_{x}\Bigl(\sum_{n=1}^\infty
\varepsilon_nV(\xi_n)=\infty\Bigr)>0\}=0,
\]  then $\{P(t)\}_{t\ge 0}$ is strongly stable.
\end{cor}
\begin{proof}
Observe that for all $x,q>0$ we have $Q^\leftarrow(Q(x)+q)< x$. Since
$H^\leftarrow_x(q)\le 1$ for all $x>0$, $q\in[0,1]$, we obtain
\[
\xi_{n}\le Q^\leftarrow(Q(\xi_{n-1})+\varepsilon_n)< \xi_{n-1},\quad n\ge 1,
\]
which shows that the sequence $\xi_n$ is decreasing. We also have
\begin{equation}\label{eq:estphi}
\frac{q}{\sup_{z\in I_{x,q}}\varphi(z)}\le G(Q^{\leftarrow}(Q(x)+q))-G(x)\le
\frac{q}{\inf_{z\in I_{x,q}}\varphi(z)},
\end{equation}
where $I_{x,q}=[Q^{\leftarrow}(Q(x)+q),x]$. The first inequality in
\eqref{eq:estphi} implies that
\[
\phi_{\xi_{k-1}}^{\leftarrow}(\varepsilon_k)\ge \frac{\varepsilon_k}{\sup_{z\le
\xi_{k-1}}\varphi(z)}\ge \frac{\varepsilon_k}{\sup_{z\le \xi_{0}}\varphi(z)},
\quad k\ge 1,
\]
which proves the first assertion, as in the proof of Corollary~\ref{cor:Mfd}. The
second statement is a direct consequence of the second inequality
in~\eqref{eq:estphi} and Corollary~\ref{cor:ss}.
\end{proof}

\begin{exmp}
Consider a homogenous kernel as in~\eqref{eq:hk}. First suppose that
$\varphi(x)\ge a/x^{\gamma}$ for $x>0$, where $a,\gamma>0$. Then $\{P(t)\}_{t\ge
0}$ is strongly stable and condition \eqref{eq:ssest} holds for every density
$u\in\mathcal{D}(A)$.

For the particular choice of $g(x)=x^{1-\beta}$ and $\varphi(x)=ax^\alpha$ for
$x>0$, condition \eqref{eq:asGQd} holds if and only if $\beta\le 0$ and
$\alpha+\beta\le 0$. Hence, if $\beta\le 0$ then $\{P(t)\}_{t\ge 0}$ is stochastic
when $0\le \alpha\le -\beta$ and it is strongly stable when $\alpha<0$.

Observe also that the case when $G(0)<\infty$, which in this example holds when
$\beta>0$, corresponds to the situation when for every $x>0$ there is $t\in
(0,\infty)$ such that  $\pi_t x=0$, so that $0$ is reached in a finite time from
every point.
\end{exmp}

\section*{Acknowledgments} This work was supported by the Natural
Sciences and Engineering Research Council (NSERC, Canada), the Mathematics of
Information Technology and Complex Systems (MITACS, Canada), and by Polish MNiSW
grant N N201 0211 33. This research was partially carried out when the author was
visiting McGill University. The author would like to thank Michael C.~Mackey for
several interesting discussions. Helpful comments of the anonymous referee are
gratefully acknowledged.

\end{document}